\documentclass[reqno]{amsart}
\usepackage{cases}
\usepackage{latexsym}
\usepackage{amsmath}
\usepackage[arrow,matrix]{xy}
\usepackage{stmaryrd}
\usepackage{amsfonts}
\usepackage{amsmath,amssymb,amscd,bbm,amsthm,mathrsfs,dsfont}

\usepackage{fancyhdr}
\usepackage{amsxtra,ifthen}
\usepackage{verbatim}

\theoremstyle{plain}
\newtheorem{theorem}{Theorem}[section]
\newtheorem{lemma}[theorem]{Lemma}
\newtheorem{proposition}[theorem]{Proposition}

\newtheorem{conjecture}[theorem]{Conjecture}

\theoremstyle{definition}
\newtheorem{definition}[theorem]{Definition}

\newtheorem{remark}[theorem]{Remark}

\begin{document}

\title[Lie Triple Derivations of Incidence Algebras]
{Lie Triple Derivations of Incidence Algebras}

\author{Danni Wang and Zhankui Xiao}

\address{Wang: School of Mathematical Sciences, Huaqiao University,
Quanzhou, Fujian, 362021, P. R. China}

\email{378126212@qq.com}

\address{Xiao: Fujian Province University Key Laboratory of Computation Science,
School of Mathematical Sciences, Huaqiao University,
Quanzhou, Fujian, 362021, P. R. China}

\email{zhkxiao@hqu.edu.cn}

\begin{abstract}
Let $\mathcal{R}$ be a $2$-torsion free commutative ring with unity, $X$ a locally finite pre-ordered set
and $I(X,\mathcal{R})$ the incidence algebra of $X$ over $\mathcal{R}$.
If $X$ consists of a finite number of connected components, we prove in this paper that
every Lie triple derivation of $I(X,\mathcal{R})$ is proper.
\end{abstract}

\subjclass[2010]{Primary 16W25, Secondary 16W10, 06A11, 47L35}

\keywords{Lie triple derivation, derivation, incidence algebra}

\thanks{This work is partially supported by the NSF of Fujian Province (No. 2018J01002)
and the National Natural Science Foundation of China (No. 11301195).}

\maketitle

\section{Introduction}\label{xxsec1}

Let $A$ be an associative algebra over $\mathcal{R}$, a commutative ring with unity. Then
$A$ has the Lie algebra structure under the Lie bracket $[x,y]:=xy-yx$. An $\mathcal{R}$-linear map
$D: A\rightarrow A$ is called a derivation if $D(xy)=D(x)y+xD(y)$ for all $x,y\in A$, and
an $\mathcal{R}$-linear map $L: A\rightarrow A$ is called a {\em Lie triple derivation} if
$$
L([[x,y],z])=[[L(x),y],z]+[[x,L(y)],z]+[[x,y],L(z)]
$$
for all $x,y,z\in A$. Let $D$ be a derivation of $A$ and $F$ be an $\mathcal{R}$-linear map
from $A$ into its centre. An observation shows that $D+F$ is a Lie triple derivation if and only if $F$ annihilates
all second commutators $[[x,y],z]$. A Lie triple derivation of the form $D+F$, with $D$ a derivation and $F$ a central-valued map,
will be called a {\em proper} Lie triple derivation. Otherwise, a Lie triple derivation will be called {\em improper}.

The problem to identify a class of algebras on which every Lie triple derivation is proper has its origin in
the Herstein's Lie-type mapping research program \cite{Her}. We refer the reader to Bre\v{s}ar's survey paper \cite{Bre04} for a comprehensive
and historic background. Miers proved that if $A$ is a von Neumann algebra with no central abelian summands,
then each Lie triple derivation of $A$ is proper \cite[Theorem 1]{Miers}. Bre\v{s}ar \cite{Bre93} extended this
result to the prime rings and, moreover, he provided a new way to study all the Lie-type maps in
the Herstein's program. On the other hand, Miers' result was extended to Lie $n$-derivations for
linear case \cite{Abdu} and nonlinear case \cite{FWX}. Here a Lie $3$-derivation means a Lie triple
derivation. Recently, many authors have made essential contributions to the related topics, see
\cite{Lu,SunMa,ZhangWuCao} for nest algebras,
\cite{JiWang} for TUHF algebras, \cite{Benk,BenkEr,JiLiuZhao,LiShen,XiaoWei1} for triangular algebras,
\cite{XiaoWei2} for full matrix algebras, etc.

The objective of this paper is to investigate Lie triple derivations on incidence algebras.
Let $(X,\leqslant)$ be a locally finite pre-ordered set.
That means $\leqslant$ is a reflexive and transitive binary relation on the set $X$, and for any
$x\leqslant y$ there are only finitely many elements $z\in X$ satisfying $x\leqslant z\leqslant y$.
The {\em incidence algebra} $I(X,\mathcal{R})$ of $X$ over $\mathcal{R}$ is defined on the set (see \cite{Kopp,SpDo})
$$
I(X,\mathcal{R}):=\{f: X\times X\longrightarrow \mathcal{R}\mid f(x,y)=0\ \text{if}\ x\nleqslant y\}
$$
with algebraic operations given by
$$\begin{aligned}
(f+g)(x,y)&=f(x,y)+g(x,y),\\
(rf)(x,y)&=rf(x,y),\\
(fg)(x,y)&=\sum_{x\leqslant z\leqslant y}f(x,z)g(z,y)
\end{aligned}$$
for all $f,g\in I(X,\mathcal{R})$, $r\in \mathcal{R}$ and $x,y\in X$. The product $(fg)$ is usually called
{\em convolution} in function theory. It is clear that the full matrix algebra ${\rm M}_n(\mathcal{R})$,
the upper (or lower) triangular matrix algebras ${\rm T}_n(\mathcal{R})$, and the infinite
triangular matrix algebras ${\rm T}_{\infty}(\mathcal{R})$ are examples of incidence algebras.

Ward \cite{Wa} firstly considered the incidence algebra of a partially ordered set (poset)
as the generalized algebra of arithmetic functions. Rota and Stanley developed incidence algebras
as the fundamental structures of enumerative combinatorics. Especially, the theory of
M\"{o}bius functions, including the classical M\"{o}bius function of number theory and the
combinatorial inclusion-exclusion formula, is established in the context of incidence algebras (see \cite{Stanley}).
Following the Stanley's work \cite{St}, automorphisms and related algebraic maps of incidence algebras have been
extensively studied (see \cite{BruL,Kopp,Sp} and the references therein).

Notice that in the theory of operator algebras, the incidence algebra $I(X,\mathcal{R})$ of a finite poset $X$ is
referred as a digraph algebra or a finite dimensional CSL algebra. Hence the second author of this note \cite{Xiao},
Khrypchenko \cite{Khry}, and
Zhang-Khrypchenko \cite{ZhangKh} studied the Herstein's program on incidence algebras
in a linear and combinatorial manner. Our main motivation of this article is,
following the trace of \cite{Xiao,Khry,ZhangKh}, to connect the Herstein's program to operator algebras
depending on the methods of linear algebra. Here we emphasize more on the combinatorial technique and
the computation is to some extent tremendous.

\section{The Finite Case}\label{xxsec2}

In this section, we study Lie triple derivations of the incidence algebra $I(X,\mathcal{R})$ when $X$ is a {\em finite} pre-ordered set.
Let's start with a proposition for general algebras. For an $\mathcal{R}$-algebra $A$, we denote by $\mathcal{Z}(A)$ the
centre of $A$ and say that $A$ satisfies the condition $(\spadesuit)$ if
$$
\mathcal{Z}(A)=\{a\,|\,[[a,x],y]=0,\ \forall x,y\in A\}.
$$

\begin{proposition}\label{sec2.1}
Let $A,B$ be two $\mathcal{R}$-algebras satisfying the condition $(\spadesuit)$.
Then $A$ and $B$ have no improper Lie triple derivations if and only if $A\oplus B$ has no improper Lie triple derivations.
\end{proposition}

\begin{proof}
We write $A\oplus B=\left[ \smallmatrix
A & 0\\
0 & B\\
\endsmallmatrix
\right]$ for convenience. Assume that $A$ and $B$ have no improper Lie triple derivations.
Let $L$ be a Lie triple derivation of $A\oplus B$. By \cite[Proposition 3.1]{XiaoWei1}, $L$ is of the form
$$
L\left[
\begin{array}
[c]{cc}%
a & 0\\
0 & b\\
\end{array}
\right]=\left[
\begin{array}
[c]{cc}%
l_A(a)+h_B(b) & 0\\
0 & h_A(a)+l_B(b)\\
\end{array}
\right],
$$
where $l_A: A\rightarrow A$, $l_B: B\rightarrow B$,
$h_A: A\rightarrow B$, $h_B: B\rightarrow A$ are linear maps satisfying
\begin{enumerate}
\item[(a)] $l_A$ is a Lie triple derivation of $A$,
$h_A([[a_1,a_2],a_3])=0$, $[[h_A(a),b_1],b_2]=0$,
for all $a_1, a_2, a_3, a\in A$ and $b_1,b_2\in B$;

\item[(b)] $l_B$ is a Lie triple derivation of $B$,
$h_B([[b_1,b_2],b_3])=0$, $[[h_B(b),a_1],a_2]=0$,
for all $a_1, a_2\in A$ and $b_1, b_2,b_3, b\in B$.
\end{enumerate}
The condition $(\spadesuit)$ implies that $h_A(a)\in \mathcal{Z}(B)$ and $h_B(b)\in \mathcal{Z}(A)$.
By the assumption, $l_A$ (resp. $l_B$) is proper. There exist a derivation $d_A$ of $A$
(resp. $d_B$ of $B$) and a central valued linear map $f_A$ (resp. $f_B$) such that
$l_A=d_A+f_A$ (resp. $l_B=d_B+f_B$). Let $D\left(\left[ \smallmatrix
a & 0\\
0 & b\\
\endsmallmatrix
\right] \right):=\left[ \smallmatrix
d_A(a) & 0\\
0 & d_B(b)\\
\endsmallmatrix
\right]$ and $F\left(\left[ \smallmatrix
a & 0\\
0 & b\\
\endsmallmatrix
\right] \right):=\left[ \smallmatrix
f_A(a)+h_B(b) & 0\\
0 & h_A(a)+f_B(b)\\
\endsmallmatrix
\right]$. Then $L=D+F$ is proper.

Conversely, if $A\oplus B$ has no improper Lie triple derivations, we need show that $A$ (and similarly $B$)
has no improper Lie triple derivations. Let $l_A$ be a Lie triple derivation of $A$. Clearly
$L\left(\left[ \smallmatrix
a & 0\\
0 & b\\
\endsmallmatrix
\right] \right):=\left[ \smallmatrix
l_A(a) & 0\\
0 & 0\\
\endsmallmatrix
\right]$ defines a Lie triple derivation of $A\oplus B$ and hence $L$ is proper.
We have $L=D+F$ with $D$ a derivation such that $D\left(\left[ \smallmatrix
a & 0\\
0 & 0\\
\endsmallmatrix
\right] \right)=\left[ \smallmatrix
d_A(a) & 0\\
0 & h_A(a)\\
\endsmallmatrix
\right]$ and $F$ a central valued linear map such that $F\left(\left[ \smallmatrix
a & 0\\
0 & 0\\
\endsmallmatrix
\right] \right)=\left[ \smallmatrix
f_A(a) & 0\\
0 & -h_A(a)\\
\endsmallmatrix
\right]$. It is straightforward to verify that $d_A$ is a derivation of $A$ and $f_A(a)\in \mathcal{Z}(A)$.
Therefore $l_A=d_A+f_A$ as desired.
\end{proof}

The condition $(\spadesuit)$ is equivalent to that there are no nonzero central inner derivations on $A$
and $B$, which has been explicitly studied in \cite[Sections 3 and 4]{XiaoWei1}. We shall show that
the incidence algebra $I(X,\mathcal{R})$ satisfies the condition $(\spadesuit)$.

Let's introduce some standard notations for the incidence algebra $I(X,\mathcal{R})$.
The unity element $\delta$ of $I(X,\mathcal{R})$ is given by $\delta(x,y)=\delta_{xy}$ for $x\leqslant y$,
where $\delta_{xy}\in \{0,1\}$ is the Kronecker delta. If $x,y\in X$ with $x\leqslant y$, let $e_{xy}$ be
defined by $e_{xy}(u,v)=1$ if $(u,v)=(x,y)$, and $e_{xy}(u,v)=0$ otherwise. Then $e_{xy}e_{uv}=\delta_{yu}e_{xv}$
by the definition of convolution. Moreover, the set $\mathfrak{B}:=\{e_{xy}\mid x\leqslant y\}$ forms an $\mathcal{R}$-linear
basis of $I(X,\mathcal{R})$ when $X$ is finite. For $i\leqslant j$ and $i\neq j$, we write $i<j$ or $j>i$ for short.

Here it is convenient to view $I(X,\mathcal{R})$ as a digraph algebra. This means that
there is a directed graph with the vertex set $X$ associated with $I(X,\mathcal{R})$. This graph contains all the
self loops and the matrix unit $e_{xy}$ corresponds to a directed edge from $y$ to $x$. The following lemma
is a little bit stronger than $(\spadesuit)$.

\begin{lemma}\label{sec2.2}
Let $X$ be finite and connected. Then there are no nonzero central derivations on $I(X,\mathcal{R})$.
\end{lemma}

\begin{proof}
Since $X$ is connected, it is well-known that $\mathcal{Z}(I(X,\mathcal{R}))=\mathcal{R}\delta$ (see \cite{SpDo} for example).
Let $D$ be a central derivation on $I(X,\mathcal{R})$. Assume $D(e_{ij})=\sum_{e_{xy}\in \mathfrak{B}}C_{xy}^{ij}e_{xy}$,
for all $e_{ij}\in \mathfrak{B}$. By \cite[Theorem 2.2]{Xiao},
$$
D(e_{ij})=\sum_{x<i}C^{ii}_{xi}e_{xj}+C^{ij}_{ij}e_{ij}+\sum_{y>j}C^{jj}_{jy}e_{iy},
$$
where the coefficients satisfy $C_{ij}^{ii}+C_{ij}^{jj}=0$ for $i\leqslant j$ and
$C_{ij}^{ij}+C_{jk}^{jk}=C_{ik}^{ik}$ for $i\leqslant j\leqslant k$.
On the other hand, the assumption implies that $D(e_{ij})=r\delta=\sum_{x\in X}re_{xx}$ is a scalar matrix. Combining the above
facts, we have $D=0$.
\end{proof}

The main result of this section is as follows.

\begin{theorem}\label{main in section 2}
Let $\mathcal{R}$ be a $2$-torsion free commutative ring with unity, and $L$ be a Lie triple derivation of $I(X,\mathcal{R})$.
Then $L$ is proper.
\end{theorem}

We only need to prove Theorem \ref{main in section 2} when $X$ is connected. In fact, assume that
$X=\bigsqcup_{i\in I}X_i$ be the union of its distinct connected components, where $I$ is a finite index set.
Let $\delta_i:=\sum_{x\in X_i} e_{xx}$. It follows from \cite[Theorem 1.3.13]{SpDo} that $\{\delta_i\mid i\in I\}$ forms
a complete set of central primitive idempotents. In other words, $I(X,\mathcal{R})=\bigoplus_{i\in I}\delta_i
I(X,\mathcal{R})$. Clearly $\delta_i I(X,\mathcal{R})\cong I(X_i,\mathcal{R})$ for each $i\in I$.
It is straightforward to verify that there are no nonzero central derivations on $I(X,\mathcal{R})$.
Hence we only need to prove Theorem \ref{main in section 2} when $X$ is connected by Proposition
\ref{sec2.1} and Lemma \ref{sec2.2}.

From now on, we assume $X$ is finite and connected until the end of this section.
Let $L: I(X,\mathcal{R})\rightarrow I(X,\mathcal{R})$ be a Lie triple derivation. We denote
for all $i,j\in X$ with $i\leqslant j$
$$
L(e_{ij})=\sum_{e_{xy}\in \mathfrak{B}}C_{xy}^{ij}e_{xy}.
$$
We make the convention $C_{xy}^{ij}=0$, if needed, for $x\nleqslant y$.

\begin{lemma}\label{sec2.4}
The Lie triple derivation $L$ satisfies
\begin{align}
L(e_{ii})&=\sum_{x<i}C_{xi}^{ii}e_{xi}+\sum_{x\in X}C_{xx}^{ii}e_{xx}+\sum_{y>i}C_{iy}^{ii}e_{iy};\label{(1)}\\
L(e_{ij})&=\sum_{x<i}C_{xi}^{ii}e_{xj}+C_{ij}^{ij}e_{ij}+\sum_{y>j}C_{jy}^{jj}e_{iy}, \hspace{6pt} \text{if}\ i\neq j.\label{(2)}
\end{align}
\end{lemma}

\begin{proof}
Without loss of generality, we assume that $|X|\geq 2$. Since $X$ is connected, each element $x\in X$ must be
a start vertex or an end vertex of a path, i.e., $x$ covers or is covered by an another element.
Let us choose an arbitrary path with the start vertex $j$ and the end vertex $i$. In other words,
$e_{ij}\in \mathcal{B}$ with $i<j$.

For the end vertex $i$, since $L(e_{ij})=L([e_{ii},[e_{ii},e_{ij}]])$, we have
$$\begin{aligned}
L(e_{ij})&=[L(e_{ii}),[e_{ii},e_{ij}]]+[e_{ii},[L(e_{ii}),e_{ij}]]+[e_{ii},[e_{ii},L(e_{ij})]]\\
&=L(e_{ii})e_{ij}-2e_{ij}L(e_{ii})+e_{ii}L(e_{ii})e_{ij}+e_{ij}L(e_{ii})e_{ii}\\
&\ +e_{ii}L(e_{ij})-2e_{ii}L(e_{ij})e_{ii}+L(e_{ij})e_{ii}.
\end{aligned}\eqno(3)$$
Since $\mathcal{R}$ is $2$-torsion free, left multiplication by $e_{ii}$ and
right multiplication by $e_{yy}$ in $(3)$ leads to
$$\begin{aligned}
C^{ii}_{ii}&=C^{ii}_{jj},\hspace{6pt} \text{if} \hspace{6pt} y=j;\\
C^{ii}_{jy}&=0,\hspace{6pt}\hspace{10pt} \text{if} \hspace{6pt} y\neq i,j.
\end{aligned}\eqno(4)$$
It follows from the relation $(4)$ that
$$\begin{aligned}
L(e_{ii})&=\sum_{e_{xy}\in \mathcal{B}}C_{xy}^{ii}e_{xy}=\sum_{x\leqslant y,x\neq j}C_{xy}^{ii}e_{xy}+C_{ji}^{ii}e_{ji}+C_{jj}^{ii}e_{jj}\\
&=\sum_{x\leqslant i,x\neq i,j}C_{xi}^{ii}e_{xi}+C_{ii}^{ii}e_{ii}+\sum_{x\leqslant y;x\neq j,y\neq i}C_{xy}^{ii}e_{xy}+C_{ji}^{ii}e_{ji}+C_{jj}^{ii}e_{jj}\\
&=\sum_{x<i}C_{xi}^{ii}e_{xi}+C_{ii}^{ii}(e_{ii}+e_{jj})+\sum_{y>i}C_{iy}^{ii}e_{iy}+\sum_{x\leqslant y;x\neq i,j,y\neq i}C_{xy}^{ii}e_{xy}.
\end{aligned}\eqno(5)$$
For any $i\neq x\in X$, from $L([[e_{ii},e_{xx}],e_{xx}])=0$ we get
$$
L(e_{ii})e_{xx}-2e_{xx}L(e_{ii})e_{xx}+e_{xx}L(e_{ii})+e_{ii}L(e_{xx})e_{xx}+e_{xx}L(e_{xx})e_{ii}=0.
$$
Multiplying the above identity by $e_{xx}$ from left and by $e_{yy}$ from right, we obtain
$$
C^{ii}_{xy}=0,\ \text{if}\ i\neq x<y\neq i.
$$
Hence the identity $(5)$ can be rewritten as
$$\begin{aligned}
L(e_{ii})&=\sum_{x<i}C_{xi}^{ii}e_{xi}+C_{ii}^{ii}(e_{ii}+e_{jj})+\sum_{y>i}C_{iy}^{ii}e_{iy}+\sum_{x\neq i,j}C_{xx}^{ii}e_{xx}\\
&=\sum_{x<i}C_{xi}^{ii}e_{xi}+\sum_{x\in X}C_{xx}^{ii}e_{xx}+\sum_{y>i}C_{iy}^{ii}e_{iy}.
\end{aligned}\eqno(6)$$

Let us now consider the start vertex $j$.
Similarly, left multiplication by $e_{xx}$ and right multiplication by $e_{jj}$ in $L(e_{ij})=L([[e_{ij},e_{jj}],e_{jj}])$ leads to
$$\begin{aligned}
C^{jj}_{jj}&=C^{jj}_{ii},\hspace{6pt} \text{if} \hspace{6pt} x=i;\\
C^{jj}_{xi}&=0,\hspace{6pt}\hspace{10pt} \text{if} \hspace{6pt} x\neq i,j.
\end{aligned}\eqno(7)$$
Then left multiplication by $e_{xx}$ and right multiplication by $e_{yy}$ in $L([e_{yy},[e_{yy},e_{jj}]])=0$ leads to
$$
C^{jj}_{xy}=0,\ \text{if}\ j\neq x<y\neq j.
$$
A similar computation shows that
$$
L(e_{jj})=\sum_{x<j}C_{xj}^{jj}e_{xj}+\sum_{y\in X}C_{yy}^{jj}e_{yy}+\sum_{y>j}C_{jy}^{jj}e_{jy}.
\eqno(8)$$
Since each element $x\in X$ must be a start vertex or an end vertex of a path, the identities $(6)$ and $(8)$ describe the
desired form of $L(e_{xx})$ for any $x\in X$.

We next describe the form of $L(e_{ij})$. It follows from equations $(1)$, $(4)$, $(7)$ that
$$\begin{aligned}
L(e_{ij})&=L\big([[e_{ii},e_{ij}],e_{jj}]\big)\\
&=[[L(e_{ii}),e_{ij}],e_{jj}]+[[e_{ii},L(e_{ij})],e_{jj}]+[e_{ij},L(e_{jj})]\\
&=[[\sum_{x<i}C_{xi}^{ii}e_{xi}+\sum_{x\in X}C_{xx}^{ii}e_{xx}+\sum_{y>i}C_{iy}^{ii}e_{iy},e_{ij}],e_{jj}]\\
&\quad+[[e_{ii},\sum_{e_{xy}\in B}C_{xy}^{ij}e_{xy}],e_{jj}]\\
&\quad+[e_{ij},\sum_{x<j}C_{xj}^{jj}e_{xj}+\sum_{x\in X}C_{xx}^{jj}e_{xx}+\sum_{y>j}C_{jy}^{jj}e_{jy}]\\
&=\sum_{x<i}C_{xi}^{ii}e_{xj}+C_{ij}^{ij}e_{ij}-(C_{ji}^{ii}+C_{ji}^{jj})e_{jj}+C_{ji}^{ij}e_{ji}+\sum_{y>j}C_{jy}^{jj}e_{iy}.
\end{aligned}\eqno(9)$$
Analogously,
$$\begin{aligned}
L(e_{ij})&=L\big([e_{ii},[e_{ij},e_{jj}]]\big)\\
&=\sum_{x<i}C_{xi}^{ii}e_{xj}+C_{ij}^{ij}e_{ij}-(C_{ji}^{ii}+C_{ji}^{jj})e_{ii}+C_{ji}^{ij}e_{ji}
+\sum_{y>j}C_{jy}^{jj}e_{iy}.
\end{aligned}\eqno(10)$$
Combining the equations $(9)$ and $(10)$ with the fact $i\neq j$, we get
$$
C_{ji}^{ii}+C_{ji}^{jj}=0. \eqno(11)
$$
Finally, a direct computation shows $0=e_{jj}L([[e_{ii},e_{ij}],e_{ij}])=-2e_{jj}L(e_{ij})e_{ij}$.
Hence $C^{ij}_{ji}=0$. Combining this fact with $(11)$, the identity $(9)$ or $(10)$ gives the desired form $(2)$.
\end{proof}

\begin{lemma}\label{sec2.5}
The coefficients $C^{ij}_{xy}$ are subject to the following relations:
$$\begin{aligned}
&{\rm (R1)}\hspace{6pt} C_{ij}^{ii}+C_{ij}^{jj}=0,  &&\text{if}\hspace{6pt} i<j;\\
&{\rm (R2)}\hspace{6pt} C_{ij}^{ij}+C_{jk}^{jk}=C_{ik}^{ik},&&\text{if}\hspace{6pt} i< j< k \hspace{6pt}\text{and}\hspace{6pt} i\neq k;\\
&{\rm (R3)}\hspace{6pt} C_{ij}^{ij}+C_{ji}^{ji}=0, &&\text{if}\hspace{6pt} i< j< i;\\
&{\rm (R4)}\hspace{6pt} C_{ii}^{ii}=C_{xx}^{ii},  &&\forall x\in X.
\end{aligned}$$
\end{lemma}

\begin{proof}
We consider the action of Lie triple derivation $L$ on the identity
$[[e_{ij},e_{kl}],e_{pq}]=\delta_{jk}(\delta_{lp}e_{iq}-\delta_{qi}e_{pl})+\delta_{li}(\delta_{qk}e_{pj}-\delta_{jp}e_{kq})$.
By Lemma \ref{sec2.4}, we need to study the following eight cases:
\begin{enumerate}
\item[(A)] $i=j, k=l, p=q$;
\item[(B)] $i=j, k=l, p\neq q$;
\item[(C)] $i=j, k\neq l, p=q$;
\item[(D)] $i=j, k\neq l, p\neq q$;
\item[(E)] $i\neq j, k=l, p=q$;
\item[(F)] $i\neq j, k=l, p\neq q$;
\item[(G)] $i\neq j, k\neq l, p=q$;
\item[(H)] $i\neq j, k\neq l, p\neq q$.
\end{enumerate}
It is clear that the case (E) (resp. the case (F)) can be calculated similarly with the case (C) (resp. the case (D)).
The case (B) can be deduced from the cases (C) and (E) by the Jacobi identity
$[[e_{ii},e_{kk}],e_{pq}]+[[e_{kk},e_{pq}],e_{ii}]+[[e_{pq},e_{ii}],e_{kk}]=0$.
Similarly, the case (G) can be deduced from the cases (D) and (F).
Therefore, we only need to study the cases (A), (C), (D) and (H).

{\bf Case} (A). If $i=j$, $k=l$, $p=q$, we assume $k=p$ to simplify the calculation. Then
$$\begin{aligned}
0&=L([[e_{ii},e_{kk}],e_{kk}])=[[L(e_{ii}),e_{kk}],e_{kk}]+[[e_{ii},L(e_{kk})],e_{kk}]\\
 &=L(e_{ii})e_{kk}-2e_{kk}L(e_{ii})e_{kk}+e_{kk}L(e_{ii})\\
 &\quad +e_{ii}L(e_{kk})e_{kk}-L(e_{kk})e_{ii}e_{kk}-e_{kk}e_{ii}L(e_{kk})+e_{kk}L(e_{kk})e_{ii}\\
&=\delta_{ik}\left(\sum_{x<i}C_{xi}^{ii}e_{xk}+\sum_{y>i}C_{iy}^{ii}e_{ky}-\sum_{x<k}C_{xk}^{kk}e_{xk}-\sum_{y>k}C_{ky}^{kk}e_{ky}\right)\\
 &\quad+\big(C_{ik}^{ii}+C_{ik}^{kk}+\delta_{ik}C_{kk}^{kk} \big)e_{ik}+\big(C_{ki}^{ii}+C_{ki}^{kk}+\delta_{ik}C_{ki}^{kk} \big)e_{ki}\\
 &\quad-\delta_{ik}\big(2C_{kk}^{ii}+2C_{ik}^{ii}+C_{ki}^{kk}+C_{ik}^{kk} \big)e_{kk}.\\
\end{aligned}\eqno(12)$$
Notice that if $i=k$ or the vertices $i,k$ are incomparable, the equation $(12)$ always holds.
If $i\neq k$ and $i,k$ are comparable, then $(12)$ is equivalent to $C_{ik}^{ii}+C_{ik}^{kk}=0$ for $i<k$ and $C_{ki}^{ii}+C_{ki}^{kk}=0$ for $k<i$,
and hence we obtain the relation (R1).

{\bf Case} (C). If $i=j$, $k\neq l$, $p=q$, we assume $i\neq k$ and $i\neq l$. Then
the formulas $(1)$ and $(2)$ imply that
$$\begin{aligned}
0&=L([[e_{ii},e_{kl}],e_{pp}])=[[L(e_{ii}),e_{kl}],e_{pp}]+[[e_{ii},L(e_{kl})],e_{pp}]\\
&=L(e_{ii})e_{kl}e_{pp}-e_{kl}L(e_{ii})e_{pp}-e_{pp}L(e_{ii})e_{kl}+e_{pp}e_{kl}L(e_{ii})\\
 &\quad +e_{ii}L(e_{kl})e_{pp}-L(e_{kl})e_{ii}e_{pp}-e_{pp}e_{ii}L(e_{kl})+e_{pp}L(e_{kl})e_{ii}\\
 &=-\delta_{ip}[(C_{li}^{ii}+C_{li}^{ll})e_{kp}+(C_{ik}^{ii}+C_{ik}^{kk})e_{pl}]+\delta_{pk}[(C_{li}^{ii}+C_{li}^{ll})e_{pi}+(C_{ll}^{ii}-C_{kk}^{ii})e_{pl}]\\
&\quad+\delta_{lp}[(C_{ik}^{ii}+C_{ik}^{kk})e_{ip}+(C_{kk}^{ii}-C_{ll}^{ii})e_{kp}].
\end{aligned}\eqno(13)$$
If $l\neq p\neq k$, then the equality $(13)$ can be rewritten as
$$
0=\delta_{ip}[(C_{li}^{ii}+C_{li}^{ll})e_{kp}+(C_{ik}^{ii}+C_{ik}^{kk})e_{pl}]. \eqno(14)
$$
Notice that when $i\neq p$, the equation $(14)$ always holds.
When $i=p$, we have $C_{li}^{ii}+C_{li}^{ll}=0$ for $l<i$ and $C_{ik}^{ii}+C_{ik}^{kk}=0$ for $i<k$.
If $p=k$, there is $l\neq p$ and $i\ne p$. Hence $(13)$ can be rewritten as
$$
0=(C_{li}^{ii}+C_{li}^{ll})e_{pi}+(C_{ll}^{ii}-C_{kk}^{ii})e_{pl},
$$
which in turn gives
$$
C_{ll}^{ii}=C_{kk}^{ii}\ \mbox{ for } k<l\mbox{ and }l\neq i\neq k, \eqno(15)
$$
and $C_{li}^{ii}+C_{li}^{ll}=0$ for $l<i$. If $p=l$, we similarly have
$$
C_{kk}^{ii}=C_{ll}^{ii}\ \mbox{ for } k<l\mbox{ and }k\neq i\neq l,
\eqno(16)$$
and $C_{ik}^{ii}+C_{ik}^{kk}=0$ for $i<k$.

Recall that $C_{ii}^{ii}=C_{jj}^{ii}$ for $i<j$ by $(4)$ and $C_{ii}^{ii}=C_{kk}^{ii}$ for $k<i$ by $(7)$.
Combining these facts with the identity $(15)$ or $(16)$, we have
$$
C_{xx}^{ii}=C_{yy}^{ii}\ \mbox{ for all } x\leqslant y, \eqno(17)
$$
The connectivity of $X$ shows that there is path from the vertex $i$ to any vertex $x\in X$.
A recursive procedure, using $(17)$, on the length of the path implies the desired relation (R4).

{\bf Case} (D). If $i=j$, $k\neq l$, $p\neq q$, there are two subcases to consider.

{\em Case D.1.} We assume $k=i$ (hence $i\neq l$), $l=p$ and $i\neq q$. Then
$$\begin{aligned}
L(e_{iq})&=L([[e_{ii},e_{il}],e_{lq}])\\
&=[[L(e_{ii}),e_{il}],e_{lq}]+[[e_{ii},L(e_{il})],e_{lq}]+[e_{il},L(e_{lq})]\\
&=\sum_{x<i}C_{xi}^{ii}e_{xq}+C_{ii}^{ii}e_{iq}-C_{ll}^{ii}e_{iq}-C_{qi}^{ii}e_{ll}\\
&\quad +C_{il}^{il}e_{iq}+C_{lq}^{lq}e_{iq}+\sum_{y>q}C_{qy}^{qq}e_{iy}-C_{qi}^{qq}e_{ll}\\
&=\sum_{x\in <i}C_{xi}^{ii}e_{xq}+(C_{il}^{il}+C_{lq}^{lq})e_{iq}+\sum_{y>q}C_{qy}^{qq}e_{iy},
\end{aligned}\eqno(18)$$
where the last identity in $(18)$ follows from $C_{ii}^{ii}=C_{ll}^{ii}$ for $i<l$ by $(4)$ and
the relation (R1). Comparing $(18)$ with the formula $(2)$ of $L(e_{iq})$, we obtain
$$
C_{il}^{il}+C_{lq}^{lq}=C_{iq}^{iq},\hspace{6pt}\mbox{ for } i< l< q \mbox{ and } i\neq q.
$$
Therefore, we obtain the relation (R2).

{\em Case D.2.} We assume $k=i$ (hence $i\neq l$), $l=p$ and $i=q$. Then
$$\begin{aligned}
L(e_{ii})-L(e_{ll})&=L([[e_{ii},e_{il}],e_{li}])\\
&=[[L(e_{ii}),e_{il}],e_{li}]+[[e_{ii},L(e_{il})],e_{li}]+[e_{il},L(e_{li})]\\
&=\sum_{x<i}C_{xi}^{ii}e_{xi}+(C_{ii}^{ii}-C_{ll}^{ii})e_{ii}+(C_{ll}^{ii}-C_{ii}^{ii})e_{ll}+C_{li}^{ii}e_{li}\\
&\quad +C_{il}^{il}e_{ii}-C_{il}^{il}e_{ll}-\sum_{y>l}C_{ly}^{ll}e_{ly}+C_{li}^{ll}e_{li}\\
&\quad +C_{li}^{li}e_{ii}+\sum_{y>i}C_{iy}^{ii}e_{iy}-\sum_{x<l}C_{xl}^{ll}e_{xl}-C_{li}^{li}e_{ll}\\
&=\sum_{x<i}C_{xi}^{ii}e_{xi}+C_{il}^{il}e_{ii}-C_{il}^{il}e_{ll}-\sum_{y>l}C_{ly}^{ll}e_{ly}\\
&\quad +C_{li}^{li}e_{ii}+\sum_{y>i}C_{iy}^{ii}e_{iy}-\sum_{x<l}C_{xl}^{ll}e_{xl}-C_{li}^{li}e_{ll},
\end{aligned}\eqno(19)$$
where the last identity in $(19)$ follows from the relations (R4) and (R1).
On the other hand, by formula $(1)$, we obtain
$$\begin{aligned}
L(e_{ii})-L(e_{ll})&=\sum_{x<i}C_{xi}^{ii}e_{xi}+\sum_{x\in X}C_{xx}^{ii}e_{xx}+\sum_{y>i}C_{iy}^{ii}e_{iy}\\
&\quad-\sum_{x<l}C_{xl}^{ll}e_{xl}-\sum_{y\in X}C_{yy}^{ll}e_{yy}-\sum_{y>l}C_{ly}^{ll}e_{ly}.
\end{aligned}\eqno(20)$$
Combining the equations $(19)$ and $(20)$, we have
$$\begin{aligned}
(C_{il}^{il}+C_{li}^{li})e_{ii}-(C_{il}^{il}+C_{li}^{li})e_{ll}
&=(C_{ii}^{ii}-C_{ii}^{ll})e_{ii}+(C_{ll}^{ii}-C_{ll}^{ll})e_{ll}\\
&\quad+\sum_{x\neq i,l}C_{xx}^{ii}e_{xx}-\sum_{y\neq i,l}C_{yy}^{ll}e_{yy}.
\end{aligned}\eqno(21)$$
Notice that $i\neq l$. Comparing the coefficients of $e_{ii}$ and $e_{ll}$ in $(21)$,
one deduces that $C_{ii}^{ii}-C_{ii}^{ll}=C_{ll}^{ll}-C_{ll}^{ii}$ for $i<l<i$.
Substituting $C_{ii}^{ll}=C_{ll}^{ll}$ and $C_{ll}^{ii}=C_{ii}^{ii}$ from the relation (R4),
we get $C_{ii}^{ii}-C_{ll}^{ll}=C_{ll}^{ll}-C_{ii}^{ii}$. Since $\mathcal{R}$ is $2$-torsion free,
$C_{ii}^{ii}=C_{ll}^{ll}$ for $i<l<i$.
Hence the coefficients of $e_{ii}$ and $e_{ll}$ of the right-hand side of $(21)$ are zero,
which yields the desired relation (R3).

{\bf Case} (H). If $i\neq j$, $k\neq l$, $p\neq q$,
we do not need to calculate since the relations (R1-R4) have been obtained and this completes the proof of the lemma.
\end{proof}

\begin{remark}\label{sec2.6}
In view of \cite[Lemma 2.4]{ZhangKh}, our Lemma \ref{sec2.5} means that every Lie triple derivation of $I(X,\mathcal{R})$
degenerates to a Lie derivation. In other words, Lemma \ref{sec2.5} can be strengthened, i.e., an $\mathcal{R}$-linear
map $L$ of $I(X,\mathcal{R})$ defined by the formulas $(1)$ and $(2)$ is a Lie triple derivation if and only if
the coefficients $C^{ij}_{xy}$ satisfy the relations (R1-R4).
\end{remark}

Notice that a direct proof of the strengthening version of Lemma \ref{sec2.5} (analogous to \cite[Lemma 2.4]{ZhangKh})
needs tedious calculation for the four cases (A,C,D,H). In our draft, it takes about 10 pages. Hence we present here
the Lemma \ref{sec2.5} for reader's convenience.

\begin{proof}[Proof of Theorem \ref{main in section 2}]
It follows from Remark \ref{sec2.6} and \cite[Theorem 2.1]{ZhangKh}.
\end{proof}

\section{The General Case}\label{xxsec3}

In this section, we study Lie triple derivations of $I(X,\mathcal{R})$ when $X$ is a locally finite pre-ordered set.
Let $\tilde{I}(X,\mathcal{R})$ be the $\mathcal{R}$-subspace of $I(X,\mathcal{R})$
generated by the elements $e_{xy}$ with $x\leqslant y$. That means $\tilde{I}(X,\mathcal{R})$
consists exactly of the functions $f\in I(X,\mathcal{R})$ which are nonzero only at a finite number of $(x,y)$.
Clearly $\tilde{I}(X,\mathcal{R})$ is a subalgebra of $I(X,\mathcal{R})$.
Hence $I(X,\mathcal{R})$ becomes an $\tilde{I}(X,\mathcal{R})$-bimodule in the natural manner.
Let $L: \tilde{I}(X,\mathcal{R})\rightarrow I(X,\mathcal{R})$ be a Lie triple derivation, i.e.
$$
L([[f,g],h])=[[L(f),g],h]+[[f,L(g)],h]+[[f,g],L(h)]
$$
for all $f,g,h\in \tilde{I}(X,\mathcal{R})$.
Observe that Lemmas \ref{sec2.4} and \ref{sec2.5} remain valid, when we replace the domain
of $L$ by $\tilde{I}(X,\mathcal{R})$. In fact, although the sums $L(e_{ij})=\sum_{x\leqslant y}C_{xy}^{ij}e_{xy}$ are now infinite,
multiplication by $e_{uv}$ on the left or on the right works as in the finite case.

Let's now recall some notations and results from \cite{ZhangKh}.
For any $f\in I(X,\mathcal{R})$ and $x\leqslant y$, {\it the restriction of $f$} to $\{z\in X\mid x\leqslant z\leqslant y\}$ is defined by
$$
f|_{x}^{y}=\sum _{x\leqslant u\leqslant v\leqslant y}f(u,v)e_{uv}
\eqno(22)$$
Observe that the sum above is finite, and hence $f|_{x}^{y}\in \tilde{I}(X,\mathcal{R})$.
The following fact is \cite[Lemma 3.3]{ZhangKh}.

\begin{lemma}\label{sec3.1}
The map $f\mapsto f|_{x}^{y}$ is an
algebra homomorphism $I(X,\mathcal{R})\rightarrow\tilde{I}(X,\mathcal{R})$.
\end{lemma}

For any $f\in I(X,\mathcal{R})$ and $x\leqslant y$, the following observation
$$
e_{xx}fe_{yy}=f(x,y)e_{xy} \eqno(23)
$$
will be extensively used.

\begin{lemma}\label{sec3.2}
Let $L$ be a Lie triple derivation of $I(X,\mathcal{R})$ and $x<y$. Then
$$
L(f)(x,y)=L(f|_{x}^{y})(x,y).
\eqno(24)$$
Moreover, if $L$ is a derivation, then $(24)$ holds for $x=y$ too.
\end{lemma}

\begin{proof}
We only need to prove the first claim by \cite[Lemma 3.4]{ZhangKh}.
It follows from $(23)$ that
$$\begin{aligned}
L(f)(x,y)&=[[e_{xx},L(f)],e_{yy}](x,y)\\
&=\big( L([[e_{xx},f],e_{yy}])+[e_{yy},[L(e_{xx}),f]]+[L(e_{yy}),[e_{xx},f]]\big)(x,y)\\
&=f(x,y)L(e_{xy})(x,y)+f(y,x)L(e_{yx})(x,y)\\
&\quad -(L(e_{xx})f)(x,y)+(fL(e_{xx}))(x,y)\\
&\quad +f(x,y)L(e_{yy})(x,x)-(fL(e_{yy}))(x,y)+f(x,x)L(e_{yy})(x,y).
\end{aligned}\eqno(25)$$
In particular,
$$\begin{aligned}
L(f|_{x}^{y})(x,y)&=f|_{x}^{y}(x,y)L(e_{xy})(x,y)+f|_{x}^{y}(y,x)L(e_{yx})(x,y)\\
&\quad -(L(e_{xx})f|_{x}^{y})(x,y)+(f|_{x}^{y}L(e_{xx}))(x,y)\\
&\quad +f|_{x}^{y}(x,y)L(e_{yy})(x,x)-(f|_{x}^{y}L(e_{yy}))(x,y)+f|_{x}^{y}(x,x)L(e_{yy})(x,y).
\end{aligned}\eqno(26)$$
By \cite[Lemma 3.2 (ii)]{ZhangKh}, the third, fourth and sixth terms of the right-hand side
of $(25)$ coincide with the corresponding terms of the right-hand side of $(26)$.
From the definition of the restriction of $f$, it is clear that
$f(x,y)=f|_{x}^{y}(x,y)$ and $f(x,x)=f|_{x}^{y}(x,x)$. Therefore, we only need to
show that the second term of the right-hand side
of $(25)$ coincides with the second term of the right-hand side of $(26)$.
In fact, if $y\nleqslant x$, both summands equal to $0$.
If $y\leqslant x$, then $x\leqslant z\leqslant y\Leftrightarrow y\leqslant z\leqslant x$,
which in turn shows $f(y,x)=f|_{y}^{x}(y,x)=f|_{x}^{y}(y,x)$.
\end{proof}

The following result is implicitly contained in \cite[Remarks 3.5 and 3.7]{ZhangKh}.

\begin{proposition}\label{sec3.3}
Every derivation from $\tilde{I}(X,\mathcal{R})$ to $I(X,\mathcal{R})$ can be
uniquely extended to a derivation of $I(X,\mathcal{R})$.
\end{proposition}

\begin{proof}
Let $D: \tilde{I}(X,\mathcal{R})\rightarrow I(X,\mathcal{R})$ be a derivation. We define
$$
\hat{D}(f)(x,y):=D(f|_x^y)(x,y)
$$
for all $f\in I(X,\mathcal{R})$, $x\leqslant y$. Then $\hat{D}$ is a linear extension of $D$ and
is a derivation of $I(X,\mathcal{R})$ by \cite[Remark 3.7]{ZhangKh}.
Let $E$ be a derivation of $I(X,\mathcal{R})$ satisfying $E(g)=D(g)$ for all $g\in \tilde{I}(X,\mathcal{R})$.
We have from Lemma \ref{sec3.2} that
$$
E(f)(x,y)=E(f|_x^y)(x,y)=D(f|_x^y)(x,y)=\hat{D}(f)(x,y)
$$
for all $f\in I(X,\mathcal{R})$ and $x\leqslant y$. Hence $E=\hat{D}$ and this completes the proof of the proposition.
\end{proof}

\begin{lemma}\label{sec3.4}
Let $X$ be connected and $L$ be a Lie triple derivation of $I(X,\mathcal{R})$. Then
$L(f)(x,x)=L(f)(y,y)$ for all $x,y\in X$.
\end{lemma}

\begin{proof}
Since $X$ is connected, we assume $x<y$ without lose of generality. Then
$$\begin{aligned}
L([e_{xy},f])(x,y)&=L([[e_{xx},e_{xy}],f])(x,y)\\
&=\big( [[L(e_{xx}),e_{xy}],f]+[[e_{xx},L(e_{xy})],f]+[e_{xy},L(f)]\big)(x,y)\\
&=f(y,y)L(e_{xx})(x,x)-(L(e_{xx})f)(y,y)-(fL(e_{xx}))(x,x)\\
&\quad +f(x,x)L(e_{xx})(y,y)+(L(e_{xy})f)(x,y)-f(x,y)L(e_{xy})(x,x)\\
&\quad -f(x,x)L(e_{xy})(x,y)+L(f)(y,y)-L(f)(x,x).
\end{aligned}\eqno(27)$$
Replacing $f$ by $f|_x^y$ in $(27)$, we have
$$\begin{aligned}
L([e_{xy},f|_x^y])(x,y)&=f|_x^y(y,y)L(e_{xx})(x,x)-(L(e_{xx})f|_x^y)(y,y)-(f|_x^yL(e_{xx}))(x,x)\\
&\quad +f|_x^y(x,x)L(e_{xx})(y,y)+(L(e_{xy})f|_x^y)(x,y)-f|_x^y(x,y)L(e_{xy})(x,x)\\
&\quad -f|_x^y(x,x)L(e_{xy})(x,y)+L(f|_x^y)(y,y)-L(f|_x^y)(x,x).
\end{aligned}\eqno(28)$$
By Lemmas \ref{sec3.1} and \ref{sec3.2},
$$
L([e_{xy},f])(x,y)=L([e_{xy},f]|_x^y)(x,y)=L([e_{xy}|_x^y,f|_x^y])(x,y)=L([e_{xy},f|_x^y])(x,y).
$$
Let's now compare the equations $(27)$ and $(28)$. Clearly, $f(x,y)=f|_{x}^{y}(x,y)$, $f(x,x)=f|_{x}^{y}(x,x)$
and $f(y,y)=f|_{x}^{y}(y,y)$. Hence the first, fourth, sixth and seventh terms of the right-hand side
of $(27)$ coincide with the corresponding terms of the right-hand side of $(28)$.
Notice that $(L(e_{xy})f)(x,y)=(L(e_{xy})f|_x^y)(x,y)$ by \cite[Lemma 3.2 (ii)]{ZhangKh}.
For the second summand, it follows from (i) and (ii) of \cite[Lemma 3.2]{ZhangKh} that
$$
(L(e_{xx})f|_x^y)(y,y)=(L(e_{xx})(f|_x^y)|_y^y)(y,y)=(L(e_{xx})f|_y^y)(y,y)=(L(e_{xx})f)(y,y).
$$
A similar procedure can be done for the third summand
of the right-hand side of $(28)$. Therefore,
$$
L(f)(y,y)-L(f)(x,x)=L(f|_x^y)(y,y)-L(f|_x^y)(x,x),
$$
the latter being zero by Lemma \ref{sec2.5}.
\end{proof}

\begin{definition}\label{sec3.5}
For any $f\in I(X,\mathcal{R})$, we define the {\it diagonal} of $f$ by
$$
f_d(x,y)=
\begin{cases}
f(x,y), & x=y,\\
0, & x\neq y.
\end{cases}
$$
\end{definition}

The main theorem of this paper is as follows.

\begin{theorem}\label{Main Theorem}
Let $X$ be connected and $\mathcal{R}$ be $2$-torsion free. Then every Lie triple derivation of $I(X,\mathcal{R})$ is proper.
\end{theorem}

\begin{proof}
Let $L$ be a Lie triple derivation of $I(X,\mathcal{R})$. Define $Z(f):=L(f)_d$ and $D(f):=L(f)-Z(f)$.
Then $Z$ is a linear map from $I(X,\mathcal{R})$ to the centre of $I(X,\mathcal{R})$ by Lemma \ref{sec3.4}.
We only need to show that $D$ is a derivation of $I(X,\mathcal{R})$.
Restricting $D$ to $\tilde{I}(X,\mathcal{R})$, we get that $D: \tilde{I}(X,\mathcal{R})\rightarrow I(X,\mathcal{R})$
is a derivation by Theorem \ref{main in section 2}. Extend $D$ to a derivation $\hat{D}$ of $I(X,\mathcal{R})$ by Proposition \ref{sec3.3}.
Notice that
$$
\hat{D}(f)(x,y)=D(f|_x^y)(x,y)=L(f|_x^y)(x,y)-L(f|_x^y)_d(x,y). \eqno(29)
$$
If $x<y$, the equation $(29)$ implies $\hat{D}(f)(x,y)=L(f|_x^y)(x,y)$, which is $L(f)(x,y)$ by Lemma \ref{sec3.2}.
In this case $L(f)_d(x,y)=0$, and hence $\hat{D}(f)(x,y)=L(f)(x,y)=D(f)(x,y)$.
If $x=y$, then the right-hand side of $(29)$ is zero. On the other hand, $D(f)(x,x)=L(f)(x,x)-L(f)_d(x,x)=0$.
Thus we get that $\hat{D}=D$ and $D$ is a derivation of $I(X,\mathcal{R})$.
\end{proof}

The reader may find that Theorem \ref{Main Theorem} can be generalized to the case when
$X$ consists of a finite number of connected components. The following conjecture is to some extent natural.

\begin{conjecture}
Let $(X,\leqslant)$ be a locally finite pre-ordered set and $\mathcal{R}$ be $2$-torsion free.
Then every Lie triple derivation of $I(X,\mathcal{R})$ is proper.
\end{conjecture}

\noindent{\bf Acknowledgements}.
The authors would like to thank the referees for their valuable comments and suggestions
which significantly helped us improve the final presentation of this paper.

\end{document}